\documentclass[11pt]{article}

\usepackage{amsmath}
\usepackage{amssymb}
\usepackage{authblk}
\usepackage{amsthm}
\usepackage{graphicx}
\usepackage{amsfonts}
\usepackage{euscript}
\usepackage[english]{babel}
\usepackage{xcolor}

\usepackage{hyperref}

\newcommand{\A}{{\cal{A}}}

\newcommand{\B}{{\cal{B}}}

\newcommand{\R}{{\cal{R}}}

\newcommand{\M}{{\cal{M}}}
\newcommand{\D}{{\cal{D}}}
\newcommand{\N}{{\cal{N}}}

\renewcommand{\H}{{\cal{H}}}
\newcommand{\K}{{\cal{K}}}
\renewcommand{\L}{{\cal{L}}}
\newcommand{\V}{{\cal{V}}}
\newcommand{\X}{{\cal{X}}}
\newcommand{\F}{{\cal{F}}}
\newcommand{\Y}{{\cal{Y}}}
\newcommand{\Z}{{\cal{Z}}}
\newcommand{\U}{{\cal{U}}}

\newcommand{\SC}{{\cal{SC}}}

\newcommand{\zv}[1]{\overset{*}{#1}}

\newcommand{\cl}[1]{\overline{#1}}

\newcommand{\sk}[2]{\langle #1, #2 \rangle}
\newcommand{\proj}{P_{\N^\bot}P_\M}

\newtheorem{theorem}{Theorem}[section]
\newtheorem{lemma}[theorem]{Lemma}
\newtheorem{corolary}[theorem]{Corollary}

\theoremstyle{definition}
\newtheorem{definition}[theorem]{Definition}
\newtheorem{rem}[theorem]{Remark}
\newtheorem{example}{Example}[section]

\def \dj {d\kern-0.4em\char"16\kern-0.1em}
\def \Dj{\mbox{\raise0.3ex\hbox{-}\kern-0.4em D}}

\title{Simultaneous extension of two bounded operators \\ between Hilbert spaces}

\author[1]{Marko S. Djiki\'c${}^{*,}$}
\author[1]{Jovana Nikolov Radenkovi\'c}
\affil[1]{\small Faculty of Sciences and Mathematics, University of Ni\v s, Vi\v segradska 33, 18000 Ni\v s, Serbia}

\date{}

\newcommand\blfootnote[1]{%
  \begingroup
  \renewcommand\thefootnote{}\footnote{#1}%
  \addtocounter{footnote}{-1}%
  \endgroup
}

\begin{document}

\maketitle

\begin{abstract}
The paper is concerned with the following question: if $A$ and $B$ are two bounded operators between Hilbert spaces $\mathcal{H}$ and $\mathcal{K}$, and $\mathcal{M}$ and $\mathcal{N}$ are two closed subspaces in $\mathcal{H}$, when will there exist a bounded operator $C:\mathcal{H}\to\mathcal{K}$ which coincides with $A$ on $\mathcal{M}$ and with $B$ on $\mathcal{N}$ simultaneously? Besides answering this and some related questions, we also wish to emphasize the role played by the class of so-called semiclosed operators and the unbounded Moore-Penrose inverse in this work. Finally, we will relate our results to  several well-known concepts, such as the operator equation $XA=B$ and the theorem of Douglas, Halmos' two projections theorem, and Drazin's star partial order.


\end{abstract}


\blfootnote{$^*$Corresponding author}
\blfootnote{E-mail addresses: \texttt{marko.djikic@gmail.com} (Marko S. Djiki\'c), \texttt{jovana.nikolov@gmail.com} (Jovana Nikolov Radenkovi\'c)}

\section{Introduction}

The problem of extending a bounded linear operator $T:\X \to \Z$ from a closed subspace $\X$ of a Banach space $\Y$ to the whole space $\tilde{T}:\Y \to \Z$ is an old and deep problem in Banach space theory. For instance, the famous Hahn-Banach theorem tells us that if $\Z=\mathbb{C}$, the extension is always possible. On the other hand, in order for every bounded operator on $\X$ to have a bounded extension on $\Y$ regardless of $\Z$, it is necessary (and sufficient) for $\X$ to be a complemented subspace of $\Y$. Therefore, according to the theorem of Lindenstrauss and Tzafriri  \cite{Lindenstrauss}, the Hilbert spaces are exactly those Banach spaces $\Y$ which have this ``unconditional extension" property. However, even in a Hilbert space the problem of a simultaneous extension of two bounded operators given on two closed subspaces seems to be non-trivial. The purpose of the present article is to study this problem.

Let $\H$ and $\K$ denote two arbitrary, not necessarily separable, complex or real Hilbert spaces, and $\M$ and $\N$ two closed subspaces of $\H$. If $A:\M\to \K$ and $B:\N\to \K$ are two bounded operators, the question is whether there exist a bounded $C:\H\to \K$ which coincides with $A$ on $\M$ and with $B$ on $\N$ simultaneously. We will give necessary and sufficient conditions for such $C$ to exist, and we will study different aspects of continuity of a simultaneous extension of $A$ and $B$. 
A particularly interesting problem is whether the operator $C_{\M,\N}$ induced on $\M+\N$ by $A$ and $B$ is a closed operator. The instrumental step in its solution is to note that $C_{\M,\N}$ belongs to the class of so-called semiclosed operators introduced by Kaufman \cite{Kaufman} (see also \cite{Caradus}), and to express it using the unbounded Moore-Penrose inverse, by virtue of a result by Corach and Maestripieri \cite{Corach3}. As we will see, even in the case $\mathrm{dim}\  \K = 1$, that is when $A$ and $B$ are two functionals, we cannot guarantee that they can be extended simultaneously to a bounded operator on the whole space. In other words, even for operators with the simplest structure, the operator $C_{\M,\N}$ can be unbounded. However, being a semiclosed operator means that $C_{\M,\N}$ is bounded on $\M+\N$ but with respect to an inner product which makes $\M+\N$ a Hilbert space continuously embedded in $\H$.

The paper is organized as follows. In Section 2 we gather the essential background needed for the development of our study. This section does not contain essentially new results, but we will include  some proofs for the reader's convenience. We start Section 3 with an example showing that our main problems are indeed meaningful, and proceed to give our first results regarding the simultaneous extension to a bounded or to a closed operator. The results of this section are based on well-known facts about unbounded operators and their adjoints, as well as on the properties of densely defined closed projections. In Section 4 the relationship with the class of semiclosed operators is explained, and in this section we give necessary and sufficient conditions for the closedness of the operator induced on $\M+\N$ by $A$ and $B$. The role that the Moore-Penrose inverse plays in our work is explained in this section. 
The last section is comprised of several miniatures, putting our results in a somewhat different perspective: we will explain the relationship with the operator equation $XA=B$ and Douglas' famous theorem, with Halmos' canonical decomposition for two subspaces, and with the theory of partial orders for Hilbert space operators initiated by Drazin. This last relationship was already suggested in \cite{Djikic}.

\section{Preliminaries}
\label{uvod}

\subsection{Notation and conventions}

Throughout this section $\H,\K$ and $\F$ will stand for arbitrary complex Hilbert spaces, and the term subspace of a Hilbert space is used for any linear manifold, not necessarily closed.
By an operator $T$ between $\H$ and $\K$ we will always mean a linear, not necessarily bounded operator with the domain $\D(T)$ which is a subspace of $\H$, and the values in $\K$. Its range and null-space are denoted by $\R(T)$ and $\N(T)$, and if $S\subseteq \K$ then $T^{-1}(S)$ denotes the set $\{h\in \D(T)\ :\ Th \in S\}$. We always assume that the domain $\D(T)$ is equipped with the norm of $\H$ when we say that $T$ is a bounded operator. The equality $T_1=T_2$ is used to specify that $\D(T_1)=\D(T_2)$ and $T_1$ and $T_2$ coincide on this domain, unless we say on which set this equality occurs. We use $T_1\subseteq T_2$ when $\D(T_1)\subseteq \D(T_2)$ and $T_1=T_2$ on $\D(T_1)$. The notation for the graph of $T$ is $\Gamma(T):=\{(h, Th)\ :\ h\in \D(T)\}\subseteq \H\times \K$. The classes $\B(\H,\K)$, or $\B(\H)\equiv \B(\H,\H)$ are comprised of bounded everywhere defined operators.

If $\M$ is a closed subspace, then $P_\M$ denotes the orthogonal projection onto $\M$, and if $M$ and $N$ are two subspaces with a trivial intersection $M\cap N=\{0\}$, their direct sum is denoted by $M\dotplus N$, while $M\oplus N$ is used only when $M$ and $N$ are orthogonal to each other. The notation $M\ominus N$ stands for $M\cap N^\bot$.

Inner product and norm are denoted by $\sk{\cdot}{\cdot}$ and $\|\cdot\|$ in every Hilbert space, and we add a subscript $\sk{\cdot}{\cdot}_\H$ only if necessary. We write $\H\times \K$ for the usual product of Hilbert spaces, that is, $\H\times \K$ is a Hilbert space with the inner product $\sk{(h_1,k_1)}{(h_2,k_2)} = \sk{h_1}{h_2} + \sk{k_1}{k_2}$. For products of Hilbert spaces $\H\times \K$, or orthogonal decompositions $\H\oplus\K$, we use a standard notation for operator matrices. Therefore, $[A\ B] : \H\times \K \to \F$ is an operator acting as $(h,k)\mapsto Ah + Bk$, while $[A\ B]^T: \F \to \H\times \K$ acts as $f\mapsto (Af,Bf)$, etc.

\subsection{The sum of two closed subspaces}

It is a well-known fact that the sum $\M+\N$ of two closed subspaces $\M,\N\subseteq \H$ is not necessarily closed. We give a fraction of well--known results regarding this sum and refer the reader to \cite[Theorem 13, Theorem 22]{D} for the proofs.

\begin{theorem}[\cite{D}]\label{06djtD} If $\M,\N\subseteq \H$ are closed subspaces, the following statements are equivalent:
\begin{itemize}
\item[(i)] $\M+\N$ is closed;
\item[(ii)] $\M^\bot + \N^\bot$ is closed;
\item[(iii)] $\R(P_{\N^\bot}P_\M)$ is closed;
\end{itemize}
\end{theorem}

Recall that $$(\M\cap \N)^\bot = \cl{\M^\bot + \N^\bot}, \quad \mbox{and} \quad (\M + \N)^\bot = \cl{\M + \N}^\bot = \M^\bot\cap \N^\bot.$$ Note that the sum $\M+\N$ can always be expressed as a sum of two closed subspaces with a trivial intersection: $\M + \N = (\M\ominus (\M\cap \N)) \dotplus \N$, or as an orthogonal sum of one closed subspace: $\M\cap \N$ and one not necessarily closed subspace: $[\M\ominus(\M\cap \N)]\dotplus [\N\ominus(\M\cap\N)]$.

\subsection{Unbounded operators and projections}

We use the terms {\it closed} and {\it closable} for an operator in the usual way: an operator $T$ between $\H$ and $\K$ is closed if $\Gamma(T)$ is a closed set of $\H\times \K$, and it is closable if $T\subseteq T_1$ for some closed operator $T_1$. The following equivalences are straightforward (see \cite{Schmudgen}):
\begin{equation}\label{djeq11}T \mbox{ is closed }\ \Leftrightarrow \ ( (x_n)_{n\in \mathbb{N}} \subseteq \D(T),\ x_n\to x\ \ \&\ \ Tx_n \to y \ \Rightarrow \ x\in \D(T)\ \ \&\ \ Tx=y. )\end{equation}
$$T \mbox{ is closable }\ \Leftrightarrow \ ( (x_n)_{n\in \mathbb{N}} \subseteq \D(T),\ x_n\to 0\ \ \&\ \ Tx_n \to y \ \Rightarrow \ y=0. )$$
If $T_i:\D(T_i)\to \K,\ \D(T_i)\subseteq \H$, $i=1,2$, then  $T_1+T_2$ stands for the operator with the domain $\D(T_1)\cap \D(T_2)$ which is defined in the obvious way, while in the case that $T_1:\D(T_1)\to \K,\ \D(T_1)\subseteq \H$ and $T_2:\D(T_2) \to \F,\ \D(T_2)\subseteq \K$, the operator $T_2T_1$ has the domain $T_1^{-1}(\D(T_2))$ and acts as a classical composition.

The adjoint $T^*$ of a densely defined operator $T:\D(T)\to \K,\ \cl{\D(T)}=\H$ is defined as usual, and we gather some well-known facts in the following lemma, for the sake of convenience.

\begin{lemma}[See \cite{Schmudgen}]\label{djl6}
Let $T_i:\D(T_i)\to \K,\ \D(T_i)\subseteq \H$, $i=1,2$, and  $T_3:\D(T_3) \to \F,\ \D(T_3)\subseteq \K$.
\begin{itemize}
  \item[1.] If $T_1$ is closed, then $T_1$ is bounded if and only if $\D(T_1)$ is closed.
  \item[2.] If $T_1$ is densely defined, then $T_1^*$ is closed. Moreover, $T_1$ is closable if and only if $\D(T_1^*)$ is dense in $\K$, and $T_1$ is bounded if and only if $\D(T_1^*)=\K$.
  \item[3.] If $T_1$ is densely defined and $T_2\in \B(\H,\K)$, then $(T_1+T_2)^*=T_1^*+T_2^*$.
  \item[4.] If $T_1,T_3$ and $T_3T_1$ are densely defined, then $(T_3T_1)^* \supseteq T_1^*T_3^*$. If $T_1$ is densely defined and $T_3\in\B(\K,\F)$, then $(T_3T_1)^*=T_1^*T_3^*$.
\end{itemize}

\end{lemma}

The statements of the following lemma will be used frequently, but the proofs are straightforward and we omit them. Statement 2 can be given in a more general form, but this form will suffice for the present work.

\begin{lemma}\label{djl9} Let $T:\D(T)\to \K,\ \D(T)\subseteq \H$.
\begin{itemize}
  \item[1.] If $A\in \B(\H,\K)$, then $T+A$ is bounded (closed, closable) if and only if $T$ is bounded (closed, closable).
  \item[2.] If $\D(T)=N \oplus \M$, where $\M$ is a closed subspace of $\H$, and $T$ is bounded on $\M$, then $T$ is bounded (closed, closable) if and only if the restriction of $T$ to $N$ is bounded (closed, closable).
  \item[3.] If $A\in \B(\F,\H)$, and $T$ is closed, then $TA$ is a closed operator.
\end{itemize}
\end{lemma}

The word {\it projection} is reserved for any operator $T:\D(T)\to \H,\ \D(T)\subseteq \H$ satisfying $\R(T)\subseteq \D(T)$ and $T^2 x = Tx$ for all $x\in \D(T)$. For a projection $T$ we have $\D(T)=\R(T)\dotplus \N(T)$, and we will use the notation $P_{\R(T)//\N(T)}$ (on the other hand, {\it orthogonal projections} denoted by $P_\M$ are always everywhere defined). Conversely, every pair of subspaces $M,N\subseteq \H$ such that $M\cap N=\{0\}$ defines a projection with the domain $M\dotplus N$, namely, $P_{M//N}$. A classical reference on the subject of unbounded projections is \cite{Ota}, see also \cite{Ando}. We gather important facts in the following lemma.

\begin{lemma}\label{djl7}
Let $T=P_{M//N}$.
\begin{itemize}
  \item[1.] $T$ is closed if and only if both $M$ and $N$ are closed.
  \item[2.] If $M$ and $N$ are closed, then $T$ is bounded if and only if $M\dotplus N$ is closed.
  \item[3.] If $T$ is closed and densely defined, then $T^*$ is also a closed densely defined projection with the range $\R(T^*)=N^\bot$ and the null-space $\N(T^*)=M^\bot$.

\end{itemize}
\end{lemma}
\begin{proof}
{\bf 1.} See \cite[Lemma 3.5]{Ota}.
\\
\\
{\bf 2.} If $M$ and $N$ are closed, $T$ is closed, and so it is bounded if and only if $\D(T)$ is closed. See also \cite[Theorem 12]{D}.
\\
\\
{\bf 3.}  In \cite[Proposition 3.4]{Ota} it is proved that $T^*$ is a projection so it is a closed projection, hence $\R(T^*)$ and $\N(T^*)$ are closed. The range and null-space relations are classical (see \cite{Schmudgen}).
\end{proof}

If $P:\H \to \H$ is an arbitrary projection, and $A:\H \to \H$ is an arbitrary operator 
we can readily check that
\begin{equation}\label{ObrisiLabelu?}\R(PA)=\R(P)\cap(\N(P)+\R(A)),\quad \N(AP)=\N(P)\dotplus(\R(P)\cap\N(A))\end{equation}
and we will use this throughout the article, without emphasizing it.

\subsection{Moore-Penrose generalized inverse}

Let $A\in \B(\H,\K)$ and denote by $A_0:\N(A)^\bot \to \R(A)$ the restriction of $A$ to $\N(A)^\bot$ and $\R(A)$, which is a bijection. {\it The Moore-Penrose inverse} of $A$, denoted by $A^\dag$, is the operator with the domain $\D(A^\dag)=\R(A)\oplus \R(A)^\bot$ which is equal to $A_0^{-1}$ on $\R(A)$ and is equal to $0$ on $\R(A)^\bot$. We refer the reader to \cite{BenIsrael} for the vast material on this subject, and we put all the properties we need in one lemma. The first four statements are contained in \cite{BenIsrael}, while the last one is a direct corollary of the closed graph theorem and Lemma \ref{djl9}

\begin{lemma}\label{djl8}
Let $A\in \B(\H,\K)$.
\begin{itemize}
  \item[1.] $A^\dag$ is a densely defined closed operator,  $\R(A^\dag)=\N(A)^\bot=\cl{\R(A^*)}$ and $\N(A^\dag)=\R(A)^\bot=\N(A^*)$.
  \item[2.] $A^\dag A = P_{\N(A)^\bot}$ and $AA^\dag = P_{\R(A)//\R(A)^\bot}$.
  \item[3.] $A^\dag$ is bounded if and only if the range $\R(A)$ is closed.
  \item[4.] $(A^*)^\dag = (A^\dag)^*$.
  \item[5.] If $B\in \B(\F, \K)$ and $\R(B)\subseteq \R(A)\oplus \R(A)^\bot$ then $A^\dag B \in \B(\F, \H)$.
\end{itemize}
\end{lemma}

\subsection{Semiclosed operators}

For a subspace $M\subseteq \H$ we say that it is an {\it operator range} if there exists a Hilbert space $\K$ and $A\in \B(\K,\H)$ such that $M=\R(A)$, or equivalently, if $M=\R(A)$ for some $A\in \B(\H)$ (recall that $\R(A)=\R((AA^*)^{1/2})$). The sum of two operator ranges is again an operator range, that is  $\R(A)+\R(B)=\R((AA^* + BB^*)^{1/2})$. The operator ranges are characterized by the following property:
$M\subseteq \H$ is an operator range if and only if it admits an inner product $\sk{\cdot}{\cdot}_M$ with respect to which it becomes a Hilbert space, and the embedding $J:M\to \H$,\ \ $J:x \mapsto x$, is a continuous operator between Hilbert spaces $(M,\sk{\cdot}{\cdot}_M)$ and $(\H,\sk{\cdot}{\cdot}_\H)$. The term {\it semiclosed subspace} is sometimes used for operator ranges, to emphasize this characterizing property.  For these and many other results the reader should address \cite{Fillmore}. 


Next we introduce the notion of \emph{semiclosed operators} and the (equivalent) notion of {\it quotients of bounded operators}. Unlike in the majority of the existing literature on the subject, our semiclosed operators act between different Hilbert spaces. However, this is only a technical difference, and all the results that we will invoke here, although proved in a different setting, remain true in our setting as well.

\begin{definition}
Let $T:\D(T)\to \K,\ \D(T)\subseteq \H$. The operator $T$ is semiclosed if $\Gamma(T)$ is a semiclosed subspace in $\H\times \K$. The set of all semiclosed operators between $\H$ and $\K$ (i.e. with domains in $\H$ and values in $\K$) is denoted by $\SC(\H,\K)$.
\end{definition}

This definition was given by Kaufman \cite{Kaufman} (cf. \cite{Caradus}), and there are several equivalent ways to define semiclosed operators. For example, $T\in \SC(\H,\K)$ if and only if $\D(T)$ is a semiclosed subspace of $\H$, and $T \in \B((\D(T),\sk{\cdot}{\cdot}'), \K)$ with respect to some (in fact, any) inner product $\sk{\cdot}{\cdot}'$ showing that $\D(T)$ is semiclosed.

\begin{lemma}[\cite{Kaufman}]\label{djl10}
If $T_1,T_2\in \SC(\H,\K)$ then $T_1+T_2\in \SC(\H,\K)$. If $T_1\in \SC(\H,\K)$ and $T_2\in \SC(\K,\F)$ then $T_2T_1\in \SC(\H,\F)$.
\end{lemma}

A particularly useful point of view on semiclosed operators is given by quotients of bounded operators.

\begin{definition}\label{djd1}
Let $T:\D(T)\to \K,\ \D(T)\subseteq \H$. The operator $T$ is the quotient of two bounded operators if there exists a Hilbert space $\F$ and $A\in\B(\F,\H)$, $B\in \B(\F,\K)$ such that $\R(A)=\D(T)$ and for every $x\in \F$ it holds: $T(Ax) = Bx$. In this case we will say that $T$ is the quotient of $B$ and $A$ and write $T=B/A$.
\end{definition}

Obviously the operators $A$ and $B$ appearing in the previous definition must satisfy $\N(A)\subseteq \N(B)$, and they are not unique. On the other hand, any two operators $A\in \B(\F,\H)$ and $B\in \B(\F,\K)$ satisfying $\N(A)\subseteq \N(B)$ induce a quotient operator $B/A$ between $\H$ and $\K$, with the domain $\D(B/A)=\R(A)$. It is not difficult to see that $T$ is a quotient if and only if $T$ is semiclosed. The definition of quotients we gave here is due to Izumino \cite{Izumino}, but several definitions of quotients can be found in the literature with insignificant differences.

Kaufman proved in \cite{Kaufman2} that every closed operator $T$ can be expressed as a quotient of bounded operators $B/A$ such that $\R(A^*)+\R(B^*)$ is closed (see also \cite{Koliha3}). In fact, it is also true that if $A$ and $B$ are given operators such that $\N(A)\subseteq \N(B)$ and if $B/A$ is a closed operator, then $\R(A^*)+\R(B^*)$ has to be closed. It seems that this implication, although mentioned in the literature (see \cite[Remark 3]{Koliha3}), is not explicitly proved. We will clarify it in the following lemma, for the readers convenience. 

\begin{lemma}
\label{djl11}
Let $A\in \B(\F,\H)$, $B\in \B(\F,\K)$ such that $\N(A)\subseteq \N(B)$.
\begin{itemize}
  \item[1.] $B/A$ is bounded if and only if $\R(B^*)\subseteq \R(A^*)$.
  \item[2.] $B/A$ is closable if and only if $(B^*)^{-1}(\R(A^*))$ is dense in $\K$.
  \item[3.] $B/A$ is closed if and only if $\R(A^*)+\R(B^*)$ is closed.
\end{itemize}
\end{lemma}
\begin{proof}
For {\bf 1.} see \cite[Remark 3]{Koliha3} (and \cite[\S 5]{Izumino}), and for {\bf 2.} see \cite[Lemma 2.3]{Izumino2}.
\\
{\bf 3.} The graph $\Gamma(B/A)=\{(Ax,Bx)\ :\ x\in \F\}=\{(Ax,Bx)\ :\ x\in \N(A)^\bot\}$ and the subspace $\R(A^*)+\R(B^*)$ are isometrically isomorphic, and so the assertion follows. An isometry can be constructed in essentially the same way as in \cite{Kaufman2}: $$J:\R(A^*)+\R(B^*) \to \Gamma(B/A),\quad J: x \mapsto (AD^\dag x, BD^\dag x),$$
where $D=(A^*A + B^*B)^{1/2}$. The case when $\R(A^*)+\R(B^*)$ is closed is covered in \cite[Theorem 1]{Kaufman2} (for a more approachable proof, see \cite[Lemma]{Koliha3}). If $\R(A^*)+\R(B^*)$ is not closed, the proof is again the same as in \cite[Lemma]{Koliha3}.
\end{proof}

From Lemma \ref{djl11} we see that if $A\in \B(\F,\H)$ and $B\in \B(\F,\K)$ satisfy $\N(A)\subseteq \N(B)$, then $\R(A^*)+\R(B^*)$ being closed should imply $(B^*)^{-1}(\R(A^*))$ being dense. A more transparent proof of this implication can be interesting on its own, so we finish this section by presenting it. With the already introduced notation we have:

\begin{lemma}\label{djl12}
If $\N(A)\subseteq \N(B)$ and $\R(A^*)+\R(B^*)$ is closed, then the subspace $(B^*)^{-1}(\R(A^*))$ is dense.
\end{lemma}
\begin{proof}
Suppose that $\R(A^*)+\R(B^*)$ is closed, which together with $\N(A)\subseteq \N(B)$ gives $\R(A^*)+\R(B^*)=\cl{\R(A^*)}$. In that case the operator $(A^*A + B^*B)^{1/2}$ has a closed range, equal to $\cl{\R(A)}$, and so does the operator $A^*A + B^*B$. Let $S=B^*B(A^*A + B^*B)^\dag A^*A$. The operator $S$ is the parallel sum of $A^*A$ and $B^*B$ (see \cite[\S 4]{Fillmore}) and $\R(S^{1/2})=\R(A^*)\cap \R(B^*)$. Note that, according to Lemma \ref{djl8}, the operator $(B^*)^\dag S^{1/2}$ belongs to $\B(\F,\K)$, and moreover $(B^*)^{-1}(\R(A^*)) = \N(B^*)\oplus \R((B^*)^\dag S^{1/2}).$ So if we prove that $\R((B^*)^{\dag} S^{1/2})$ is dense in $\cl{\R(B)}$, we will finish the proof. Observe that $\R((B^*)^\dag S^{1/2})\supseteq \R((B^*)^{\dag} S) = \R( B(A^*A + B^*B)^{\dag}A^*A)$, furthermore $\R((A^*A + B^*B)^{\dag} A^*A)$ is a dense set in $\cl{\R(A^*)}$, since $(A^*A + B^*B)^{\dag}$ acts as a homeomorphism on $\cl{\R(A^*)}$, while the range of $A^*A$ is dense in $\cl{\R(A^*)}$. Keeping in mind that $B(\cl{\R(A^*)})=\R(B)$, we conclude that $B(\R((A^*A + B^*B)^{\dag} A^*A))$ is dense in $\cl{\R(B)}$, and so is $\R((B^*)^{\dag} S^{1/2})$.
\end{proof}

\section{Continuity properties of $C_{\M,\N}(A,B)$}\label{djsec1}

Let $A,B\in \B(\H,\K)$ be two bounded operators between Hilbert spaces $\H$ and $\K$, and $\M,\N\subseteq \H$ two closed subspaces in $\H$. For $A$, $B$, $\M$, and $\N$, we denote by $C_{\M,\N}(A,B)$ the operator with the domain $\M+\N$ which coincides with $A$ on $\M$, and with $B$ on $\N$ simultaneously. It is clear that $C_{\M,\N}(A,B)$ exists if and only if $A$ and $B$ coincide on $\M\cap \N$. The main purpose of this work is to study continuity properties of the operator $C_{\M,\N}(A,B)$. We are concerned with the following questions.

\begin{itemize}
  \item Is $C_{\M,\N}(A,B)$ bounded? In other words, does there exist a bounded $C:\H \to \K$ which coincides with $A$ on $\M$ and with $B$ on $\N$?
  \item Is $C_{\M,\N}(A,B)$ closable? In other words, does there exist a closed $C:\D(C)\to \K$, such that $C_{\M,\N}(A,B)\subseteq C$?
  \item Is $C_{\M,\N}(A,B)$ a closed operator?
\end{itemize}

If $\M+\N$ is a closed subspace, all these questions have a trivial positive answer (see also \cite[Proposition 2.1]{Djikic}), given that $C_{\M,\N}(A,B)$ is bounded for any $A$ and $B$ which coincide on $\M\cap \N$. If $\M+\N$ is not closed, we demonstrate by the following example that these questions are meaningful, and that in general the operator $C_{\M,\N}(A,B)$ is not even closable. 

\begin{example}\label{06djex1}

Let $\K=\ell^2(\mathbb{N})$ and $T_\alpha:\K \to \K$ be defined by $$T_\alpha : (a_1,a_2,a_3,...) \to (\frac{1}{1^\alpha}a_1,\frac{1}{2^\alpha}a_2,\frac{1}{3^\alpha}a_3,...).$$ For every $\alpha>0$ the operator $T_\alpha$ is a self-adjoint, compact, and injective operator, having a dense range $\R(T_\alpha)\not = \cl{\R(T_\alpha)} = \K$. Let $\H=\K\times \K$, $\M=\K \times \{0\}$, and $\N = \Gamma(T_1) = \{(k,T_1k)\ :\ k\in \K\}$. In that case $\M$ and $\N$ are closed, $\M\cap \N=\{0\}$ and $\M\dotplus \N = \K \times \R(T_1)$ is not a closed subspace.

First let us consider $C_{\M,\N}(A,B):\M + \N \to \H$ induced by $A:(x,y)\mapsto (T_\alpha x, 0)$ and $B=0$ for different $\alpha\geq 0$. We will abbreviate $C_{\M,\N}(A,B)$ to $C_{\M,\N}$.  
\\
\\
{\bf 1.} $\alpha = 0$. In this case, $A=I$ on $\M$ and $C_{\M,\N}=P_{\M//\N}$. Therefore, $C_{\M,\N}$ is a closed operator, but $\M+\N$ is not closed, hence $C_{\M,\N}$ is unbounded (see Lemma \ref{djl7}).
\\
\\
{\bf 2.} $0<\alpha<1$. We will show that $C_{\M,\N}$ is closable, but not closed and not bounded. Assume that $((x_n,0))_{n\in \mathbb{N}}\subseteq \M$ and $((y_n,T_1 y_n))_{n\in \mathbb{N}}\subseteq \N$ such that $(x_n,0) + (y_n,T_1y_n) \to 0$ and $C_{\M,\N}((x_n,0) + (y_n,T_1y_n)) \to z=(z_1,z_2)$, for some $z\in \H$. Then we have also $T_1 x_n \to 0$, and from the definition of $C_{\M,\N}$, $z_2=0$ and $T_\alpha x_n \to z_1$. Since $T_1 = T_{1-\alpha}T_\alpha$, we conclude $T_{1-\alpha} z_1 = 0$ and so $z_1=0$, showing that $z$ must be $0$. Hence, $C_{\M,\N}$ is closable.

    To see that $C_{\M,\N}$ is not bounded, pick $\beta$ such that $\frac{1}{2} < \beta < \frac{3}{2}-\alpha$, and let $x_n = (\frac{1}{1^{\beta -1}}, \frac{1}{2^{\beta-1}},...,\frac{1}{n^{\beta-1}},0,0,...)$.
    The sequence $(x_n)_{n\in \mathbb{N}}$ is unbounded in $\ell^2$, as well as the sequence $(T_\alpha x_n)_{n\in \mathbb{N}}$, while $T_1 x_n \to (\frac{1}{n^\beta})_{n\in \mathbb{N}}$. Let $s_n = \|T_\alpha x_n\|^{-1}$  and take a sequence from $\M+\N$ defined as: $(-s_n x_n,0) + (s_n x_n, s_n T_1 x_n)$. Such a sequence converges to $0$ while $C_{\M,\N}(( -s_n x_n,0) + (s_n x_n, s_n T_1 x_n)) = -s_n (T_\alpha x_n,0)$ is unimodular. %

    Finally to see that $C_{\M,\N}$ is not closed, pick $\beta$ such that $\frac{1}{2}-\alpha < \beta \leq \frac{1}{2}$ and let $x_n = (\frac{1}{1^\beta}, \frac{1}{2^\beta},...,\frac{1}{n^\beta},0,0,...)$. We have $(-x_n,0)+(x_n,T_1x_n)\to (0, (\frac{1}{n^{\beta +1}})_{n\in \mathbb{N}})\not \in \M+\N$, while $C_{\M,\N}((-x_n,0)+(x_n,T_1x_n))=-(T_{\alpha}x_n,0)$ also converges, to $((\frac{1}{n^{\alpha+\beta}})_{n\in \mathbb{N}},0)$. Hence $C_{\M,\N}$ is not closed.
\\
\\
{\bf 3.} $\alpha\geq 1$. In this case $C_{\M,\N}$ is bounded. Namely, if $(x_n,0)+(y_n,T_1y_n)$ is a sequence from $\M+\N$ converging to $0$, then $T_1x_n \to 0 $ also, and so we have $C_{\M,\N}((x_n,0)+(y_n,T_1y_n)) = (T_\alpha x_n, 0) = (T_{\alpha - 1} T_1 x_n,0) \to 0$.
\\
\\
Let us show that the transformation $C_{\M,\N}$ can be non-closable even for finite rank $A$ and $B$.
\\
\\
{\bf 4.} Let $y=(\frac{1}{n})_{n\in \mathbb{N}}$ and $x_n = (1,1,...,1,0,0,...)$ the sequence starting with $n$ $1$'s followed by $0$'s. If $A=P_{(y,0)}$ the orthogonal projection onto the subspace spanned by $(y,0)$, and $B=0$, then $C_{\M,\N}$ does not have a closed extension on $\H$. To see this let $s_n = (\sum_{i=1}^{n} \frac{1}{i})^{-1}$ and take $\alpha_n = s_n(x_n,0)\in \M$ and $\beta_n = s_n(-x_n,-T_1x_n) \in \N$. Then $\alpha_n + \beta_n = s_n(0,-T_1x_n) \to 0$. On the other hand, $C_{\M,\N}(\alpha_n + \beta_n) = A \alpha_n = s_n(\sk{(x_n,0)}{(y,0)}_{\H} /\|(y,0)\|^2)\cdot (y,0) = \|y\|^{-2}(y,0)$, thus the sequence $(C_{\M,\N}(\alpha_n + \beta_n))$ converges to $\|y\|^{-2}(y,0)\not = 0$. This shows that $C_{\M,\N}$ does not have a closed extension. \hfill $\diamond$

\end{example}

For the rest of this section we fix $\H$ and $\K$ to be arbitrary Hilbert spaces (not necessarily separable), $\M$ and $\N$ two closed subspaces of $\H$, and $A, B \in \B(\H,\K)$ two bounded operators. We are not assuming that $A$ and $B$ coincide on $\M\cap \N$, and when they do, that is when $C_{\M,\N}(A,B)$ is well-defined, we will abbreviate it to $C_{\M,\N}$. We proceed to give necessary and sufficient conditions for the operator $C_{\M,\N}$ to be well-defined and bounded and for it to be  well-defined and closable. We will postpone giving necessary and sufficient conditions for the operator $C_{\M,\N}$ to be closed until the next section.

\begin{lemma}\label{djl1} Let $Q$ be the projection onto $\M\ominus (\M\cap \N)$ parallel with $\N\oplus (\M+\N)^\bot$. In that case, $Q$ is closed densely defined, and if $C_{\M,\N}$ exists, then for every $\alpha \in \M + \N$: $C_{\M,\N}\alpha = (A-B)Q\alpha + B\alpha$. Consequently, if $\H=\cl{\M+\N}$, then \begin{equation}\label{06djeq3}C_{\M,\N}= (A-B)Q + B.
\end{equation}
\end{lemma}
\begin{proof}
The proof is straightforward if we note  $(A-B)Q+B=AQ + B(I-Q)$, which is exactly the operator equal to $A$ on $\M$ and to $B$ on $\N\oplus(\M+\N)^\bot$. In the case when $\H=\cl{\M+\N}$, the domains of both operators $C_{\M,\N}$ and $Q$ are equal to $\M+\N$, so equality \eqref{06djeq3} holds.
\end{proof}


\begin{theorem}\label{06djt1}

 \begin{itemize}
 \item[1.] There exists a bounded operator $C:\H\to \K$ which coincides with $A$ on $\M$ and with $B$ on $\N$ if and only if $\R(A^* - B^*) \subseteq \M^\bot + \N^\bot$.
 \item[2.] There exists a closed $C:\D(C)\to \K$, $\D(C)\subseteq \H$ such that $\M+\N\subseteq \D(C)$ and that $C$ coincides with $A$ on $\M$ and with $B$ on $\N$ if and only if $(A^*-B^*)^{-1}(\M^\bot + \N^\bot)$ is dense in $\K$.
 \end{itemize}
\end{theorem}
\begin{proof}
{\bf 1.} Note that if such $C$ exists then $A$ and $B$ must coincide on $\M\cap \N$, but also if $\R(A^* - B^*) \subseteq \M^\bot + \N^\bot$, then $\M\cap \N \subseteq \N(A-B)$ showing that $A$ and $B$ coincide on $\M\cap \N$. Therefore, we can assume that $A$ and $B$ coincide on $\M\cap \N$, and consider the linear transformation $C_{\M,\N}$ as before. There is a bounded operator $C$ defined on the whole space with the given properties if and only if $C_{\M,\N}$ is bounded. Let us show that $C_{\M,\N}$ is bounded if and only if $\R(A^* - B^*) \subseteq \M^\bot + \N^\bot$.

As in Lemma \ref{djl1}, let $Q$ denote the projection onto $\M\ominus (\M\cap \N)$ parallel with $\N\oplus (\M+\N)^\bot$, so the following equality holds on $\M+\N$: $C_{\M,\N} = (A-B)Q + B$. The operator $C_{\M,\N}$ is bounded if and only if the operator $(A-B)Q+B$ is bounded on $\M+\N$. By Lemma \ref{djl9} this is equivalent to $(A-B)Q$ being bounded on its domain. It remains to prove that this happens if and only if  $\R(A^* - B^*) \subseteq \M^\bot + \N^\bot$.

Let $T=(A-B)Q$. Since $T$ is densely defined, $T$ is bounded if and only if the domain of its adjoint $T^*$ is the whole $\K$, but $T^* = Q^*(A-B)^*$ (Lemma \ref{djl6}), and so $\D(T^*)=\{k\in \K\ :\ (A^*-B^*)k \in \D(Q^*)\}$. The operator $Q^*$ is again a projection with the domain $\D(Q^*)=\R(Q^*)\dotplus \N(Q^*)=\N(Q)^\bot \dotplus \R(Q)^\bot = (\M^\bot + \N^\bot)\oplus (\M^\bot + \N^\bot)^\bot$ (Lemma \ref{djl7}). Given that $A$ and $B$ coincide on $\M\cap \N$, or in other words: $\M\cap \N \subseteq \N(A-B)$, we obtain $\R(A^*-B^*)\bot (\M^\bot+\N^\bot)^\bot$. Therefore, for $k\in \K$, $(A^*-B^*)k \in \D(Q^*)$ if and only if $(A^*-B^*)k \in \M^\bot + \N^\bot$, so we in fact have $\D(T^*) = \{k\in \K\ :\ (A^*-B^*)k \in \M^\bot + \N^\bot\}$. Finally, we conclude that $\D(T^*)=\K$ if and only if $\R(A^* - B^*) \subseteq \M^\bot + \N^\bot$, completing the proof of this assertion.
\\
\\
{\bf 2.} Similarly as in 1. first we note that both conditions imply $A$ and $B$ coinciding on $\M\cap \N$, so we can consider the transformation $C_{\M,\N}$ defined on $\M+\N$. We should in fact prove that $C_{\M,\N}$ is closable iff $\{k \in \K\ :\ (A^*-B^*)k \in \M^\bot + \N^\bot\}$ is dense in $\K$. With the analogous reasoning as in 1. we come to a conclusion that $C_{\M,\N}$ is closable iff $T$ is closable, i.e. iff $\D(T^*)$ is dense in $\K$ (Lemma \ref{djl6}). As we have already explained, $\D(T^*) = \{k \in \K\ :\ (A^*-B^*)k \in \M^\bot + \N^\bot\}$, which proves the assertion.
\end{proof}

\begin{rem}\label{06djr1}
\textbf{1.}\  If $\M+\N$ is a closed subspace, we know that the transformation $C_{\M,\N}(A,B)$ is bounded for any $A$ and $B$ which coincide on $\M\cap \N$. This is also contained in Theorem \ref{06djt1}: if $A$ and $B$ coincide on $\M\cap \N$, then $\N(A-B)\supseteq \M\cap \N$ which leads us to $\R(A^* - B^*)\subseteq \cl{\M^\bot + \N^\bot}$; if $\M+\N$ is closed, then $\cl{\M^\bot + \N^\bot}=\M^\bot + \N^\bot$ (Theorem \ref{06djtD}), so the condition of Theorem \ref{06djt1} is fulfilled.

\textbf{2.}\  The symmetry of the condition in Theorem \ref{06djt1} indicates that $C_{\M,\N}(A,B)$ is bounded if and only if $C_{\M,\N}(B,A)$ is bounded, which can seem as a curious conclusion. However, this can be proved directly as follows. The operator $C_{\M,\N}(A,B)$ is bounded iff for any two sequences $(x_n)_{n\in \mathbb{N}}\subseteq \M$ and $(y_n)_{n\in \mathbb{N}}\subseteq \N$ such that $x_n + y_n \to 0$, the sequence $(Ax_n + By_n)_{n\in \mathbb{N}}$ also converges to $0$. The analogous statement holds for $C_{\M,\N}(B,A)$ as well. Note that, if $x_n+y_n\to 0$, then $(A+B)x_n + (A+B)y_n\to 0$, so $Ax_n + By_n \to 0$ iff $Bx_n + Ay_n \to 0$, i.e. $C_{\M,\N}(A,B)$ is bounded iff $C_{\N,\M}(B,A)$ is bounded.

\textbf{3.}\ The operator $C_{\M,\N}(A,B)$ does not depend on the way $A$ and $B$ act outside $\M$ and $\N$ respectively, in fact $C_{\M,\N}(A,B)=C_{\M,\N}(AP_\M,BP_\N)$. 
However, it is perhaps not so explicit that the condition $\R(A^*-B^*)\subseteq \M^\bot + \N^\bot$ (or the one for $C_{\M,\N}$ to be closable) is also independent on the way $A$ and $B$ act outside $\M$ and $\N$. Indeed: $(AP_\M)^* - (BP_\N)^* = (A^* - B^*) + (-P_{\M^\bot}A^* + P_{\N^\bot}B^*)$, but $\R(-P_{\M^\bot}A^* + P_{\N^\bot}B^*)\subseteq \M^\bot + \N^\bot$, and so $\R((AP_\M)^* - (BP_\N)^*)\subseteq \M^\bot + \N^\bot$ if and only if $\R(A^*-B^*)\subseteq \M^\bot + \N^\bot$. \hfill $\diamond$
\end{rem}

In the following theorem we prove that two subspaces $\M$ and $\N$ have the simultaneous extension property for any two operators $A$ and $B$ if and only if $\M+\N$ is closed. In fact, we will prove that if $\M+\N$ is not closed, then for any $\K\not = \{0\}$ we can construct operators $A,B\in \B(\H,\K)$ such that $C_{\M,\N}$ is not bounded, which extends  \cite[Proposition 2.1]{Djikic}. Note also that ``bounded" can be changed to ``closable", and so also to ``closed", and the theorem still holds.

\begin{theorem}\label{jt8}
Let $\K\not = \{0\}$. The operator $C_{\M,\N}(A,B)$ is bounded for any  $A, B\in \B(\H, \K)$ coinciding on $\M\cap \N$ if and only if $\M+\N$ is a closed subspace.
\end{theorem}
\begin{proof}
Suppose that $\M+\N$ is not closed. Then $\M^\bot + \N^\bot$ is also not closed (Theorem \ref{06djtD}), so we can pick $x\in \cl{\M^\bot + \N^\bot}\setminus \M^\bot + \N^\bot$. Let $0\not = y\in \K$ and let $T:\K \to \H$ be a rank-one bounded operator mapping $y$ to $x$. Define $A=T^*P_\M$ and $B=T^*P_{\M\cap \N}$. These two operators coincide on $\M\cap \N$, and  $\R(A^*-B^*)=\R(P_\M [(I-P_{\M\cap \N}) T]) = \R(P_\M T)$, which is the one-dimensional subspace spanned by $P_\M x$. Since $P_\M x \not \in \M^\bot + \N^\bot$, we in fact have $(A^*-B^*)^{-1}(\M^\bot + \N^\bot) = \{0\}$ and so, according to Theorem \ref{06djt1}, $C_{\M,\N}(T^*P_\M, T^*P_{\M\cap \N})$ is not bounded. The opposite implication is clear.
\end{proof}


The subject of the next theorem is analogous to the previous one, just with fixed operators instead of subspaces. If $\H$ is finite--dimensional, then $C_{\M,\N}(A,B)$ is bounded for any feasible $\M$ and $\N$. If $\H$ is infinite--dimensional, then we can  find $\M$ and $\N$ such that $C_{\M,\N}(A,B)$ is not bounded, unless $A=B$.

\begin{theorem}\label{jt3}
Let $A,B\in \B(\H,\K)$ and $\H$ be infinite--dimensional. The operator $C_{\M,\N}(A,B)$ is bounded for any two subspaces $\M$ and $\N$ such that $A$ and $B$ coincide on $\M\cap \N$ if and only if $A=B$. 
\end{theorem}
\begin{proof}


Suppose that $A-B\not = 0$ and let $0\not = x\in \R(A^* - B^*)$. Let $\U,\V\subseteq \H$ be two closed subspaces such that $\U+\V \not = \cl{\U + \V} = \H$ and $x\not \in \U+\V$. Such subspaces certainly exist: first of all in any infinite--dimensional Hilbert space, and so also in $\H$, there are two closed subspaces $\U'$ and $\V'$ such that $\U'+\V'\not = \cl{\U' + \V'} = \H$ (see Example \ref{06djex1} for a separable Hilbert space); if $x\not \in \U'+\V'$ take $\U=\U'$, $\V=\V'$, and we are done; otherwise, pick $x'\not \in \U'+\V'$ such that $\|x\|=\|x'\|$ and let $\U=\varphi(\U')$ and $\V=\varphi(\V')$, where $\varphi:\H\to \H$ is an isometric isomorphism such that $\varphi(x')=x$. If we put $\M=\U^\bot$ and $\N=\V^\bot$, then $\M\cap \N=\{0\}$ and so $A$ and $B$ coincide here, but $\R(A^*-B^*)\not \subseteq \M^\bot + \N^\bot$, thus $C_{\M,\N}(A,B)$ is not bounded.  
\end{proof}



\section{$C_{\M,\N}(A,B)$ as a semiclosed operator}

We begin this section with a simple proof that $C_{\M,\N}(A,B)$ is a semiclosed operator. The construction of a new semiclosed operator from two such operators described in the following lemma seems to be new.

\begin{lemma}\label{djl2}
Let $\H$ and $\K$ be Hilbert spaces and $T_1,T_2\in \SC(\H,\K)$ be such that $T_1$ and $T_2$ coincide on $\D(T_1)\cap \D(T_2)$. Then the operator $T:\D(T_1)+\D(T_2)\to \K$ coinciding with $T_1$ on $\D(T_1)$ and with $T_2$ on $\D(T_2)$ is a semiclosed operator. Consequently, $C_{\M,\N}(A,B)$ is a semiclosed operator for any two closed subspaces $\M,\N\subseteq \H$ and any two operators $A,B\in \B(\H,\K)$ which coincide on $\M\cap \N$.
\end{lemma}

\begin{proof}
Observe that $\Gamma(T)=\Gamma(T_1)+\Gamma(T_2)$, and since the sum of two semiclosed subspaces in $\H\times \K$ is again a semiclosed subspace in $\H\times \K$, the operator $T$ is also semiclosed. The part of the statement regarding $C_{\M,\N}(A,B)$ follows directly from the observation that $A|_\M, B|_\N \in \SC(\H,\K)$.
\end{proof}


%
In order to express $C_{\M,\N}(A,B)$ in a quotient-like form, we will first explain a natural connection between the quotients as introduced in Definition \ref{djd1} and the Moore-Penrose generalized inverse. Namely, several definitions of quotients appear in the literature (see \cite{Koliha3}), but we prefer using the form $BA^\dag$. All the appearing forms (including ours) have insignificant technical differences (operators $B/A$ and $BA^\dag$ are essentially the same, except for $BA^\dag$ being defined for any $A$ and $B$ and having a dense domain). However, the use of the Moore-Penrose inverse can lead a bit further, considering the  majority of accessible formulas regarding this inverse. Using a recent result we will in fact see that Lemma \ref{djl1}  has already given us a quotient form for $C_{\M,\N}(A,B)$. First  we gather some properties of $BA^\dag$.


\begin{lemma}\label{djl3}
Let $\H,\K$ and $\F$ be Hilbert spaces and $A\in \B(\F,\H)$, $B\in\B(\F,\K)$.
\begin{itemize}
\item[1.] $BA^\dag \in \SC(\H,\K)$.
\item[2.] If $\N(A)\subseteq \N(B)$ then $BA^\dag = B/A$ on $\R(A)$.
\item[3.] $BA^\dag = (BP_{\N(A)^\bot})A^\dag$, the quotient $BP_{\N(A)^\bot} / A$ is well-defined, and the following equality holds $BA^\dag = BP_{\N(A)^\bot} / A$ on $\R(A)$.
\item[4.] $BA^\dag$ is bounded if and only if $ P_{\N(A)^\bot}(\R(B^*))\subseteq \R(A^*)$; $BA^\dag$ is closable if and only if $(P_{\N(A)^\bot}B^*)^{-1}(\R(A^*))$ is dense in $\K$; $BA^\dag$ is closed if and only if $\R(A^*) + P_{\N(A)^\bot}(\R(B^*))$ is closed.
\item[5.] $BA^\dag=\B/ \A$, where $\B$ and $\A$ are the following bounded operators: $\A = [A_0\ J]:\N(A)^\bot \times \R(A)^\bot\to \H$ and $\B = [B_0\ 0] : \N(A)^\bot \times \R(A)^\bot \to \K$, where $A_0$ and $B_0$ are the restrictions of $A$ and $B$ to $\N(A)^\bot$ respectively, while $J:\R(A)^\bot \to \H$ is the inclusion map.
\end{itemize}
\end{lemma}
\begin{proof}
{\bf 1.} Both $B$ and $A^\dag$ are semiclosed operators, so the operator $BA^\dag$ is also semiclosed, according to Lemma \ref{djl10}.
\\
\\
{\bf 2.} This is straightforward.
\\
\\
{\bf 3.} The equality $BA^\dag = (BP_{\N(A)^\bot})A^\dag$ follows from $\R(A^\dag)=\N(A)^\bot$. For the rest of the statement apply statement 2. to operators $BP_{\N(A)^\bot}$ and $A$.
\\
\\
{\bf 4.} The operator $BA^\dag$ is bounded (closable, closed) if and only if its restriction to $\R(A)$ is a bounded (closable, closed) operator, see Lemma \ref{djl9}. Therefore, using statement 3. we obtain that $BA^\dag$ has the desired property if and only if $ BP_{\N(A)^\bot} / A$ has such property. The assertion now follows from Lemma \ref{djl11}.
\\
\\
{\bf 5.} The quotient $\B/\A$ is well-defined because $\N(\A)$ is trivial, and regarding the domain we have $\D(\B/\A)=\R(A)\oplus \R(A)^\bot = \D(BA^\dag)$. Let $x\in \F$ and $h\in \R(A)^\bot$ be arbitrary. On one hand we have $BA^\dag(Ax + h) = BP_{\N(A)^\bot} x$, and on the other $(\B/\A) (Ax + h) = (\B/\A )(\A (P_{\N(A)^\bot} x, h)) = BP_{\N(A)^\bot} x$. Hence, $\B/\A = BA^\dag$.
\end{proof}

In Lemma \ref{djl1} we showed the following equality on $\M+\N$: $C_{\M,\N}(A,B) = (A-B)Q + B$, where $Q$ is the projection onto $M\ominus(\M \cap \N)$ parallel with $\N\oplus (\M+\N)^\bot$. The projection $Q$ can in fact be expressed as the Moore-Penrose inverse of a product of two orthogonal projections $P_{\N^\bot}P_\M$. The following is a result due to Corach and Maestripieri \cite{Corach3}, and we only include the proof of the part regarding range and null-space for convenience.

\begin{theorem}[See \cite{Corach3}]\label{06djt2} If $\M$ and $\N$ are closed subspaces of a Hilbert space $\H$, then $T=( P_{\N^\bot} P_\M)^\dag$ is a closed densely defined projection with the domain $\D(T)=(\M+\N)\oplus (\M+\N)^\bot$, and the range and null-space $\R(T)=\M \ominus (\M\cap \N)$ and $\N(T)=\N \oplus (\M + \N)^\bot$.
\end{theorem}
\begin{proof}
The fact that $T$ is a projection follows from \cite[Theorem 6.2]{Corach3}. If $S=P_{\N^\bot} P_\M$, since $T=S^\dag$, $\D(T)=\R(S)\oplus \R(S)^\bot$ and $\R(T)=\cl{\R(S^*)}=\N(S)^\bot$. Given that $S$ is the product of projections, we have $\R(S)=\N^\bot \cap (\M + \N)$, $\R(S)^\bot=\N(S^*)=\N \oplus (\N^\bot\cap \M^\bot) = \N\oplus (\M+\N)^\bot$, and so $\D(T)=[\N^\bot \cap (\M+\N)]\oplus \N \oplus (\M+\N)^\bot=(\M+\N)\oplus(\M+\N)^\bot$. Furthermore, $\R(T)=\N(S)^\bot = (\M^\bot \oplus (\M \cap \N))^\bot = \M \cap (\M\cap \N)^\bot$, and $\N(T)=\R(S)^\bot = \N\oplus (\M^\bot \cap \N^\bot)$. Thus: $\R(T)=\M \ominus (\M\cap \N)$ and $\N(T)=\N\oplus(\M + \N)^\bot$.
\end{proof}

From Lemma \ref{djl1} and Theorem \ref{06djt2} we obtain the following quotient-like expression for $C_{\M,\N}(A,B)$. Observe that the Moore-Penrose inverse $(\proj)^\dag$ is bounded if and only if $\M+\N$ is closed. This is suggested by its domain in Theorem \ref{06djt2}, or by looking at the range of $\proj$ and Theorem \ref{06djtD}.

\begin{lemma}\label{djl4}
Let $\H$ and $\K$ be Hilbert spaces, $A,B\in \B(\H,\K)$, and $\M,\N$ two closed subspaces of $\H$. If $A$ and $B$ coincide on $\M\cap \N$ then $C_{\M,\N}(A,B) = (A-B)(\proj)^\dag + B$ on $\M+\N$.
\end{lemma}

Combining Lemma \ref{djl4} and Lemma \ref{djl3} we can find necessary and sufficient conditions for the operator $C_{\M,\N}(A,B)$ to be closed. We begin with a geometrical lemma.

\begin{lemma}\label{djl5} Let $\H$ be a Hilbert space and $S,R\subseteq \H$ two subspaces of $\H$ such that $\cl{R}\subseteq \cl{S}$. Let $\V\subseteq \H$ be a closed subspace of $\H$ such that $\V^\bot \subseteq S$. Then $S+R$ is a closed subspace if and only if $\V\cap S + P_{\V\cap \cl{S}} (R)$ is closed.
\end{lemma}
\begin{proof}
Assume that $S+R$ is closed. In that case $\cl{S}\subseteq \cl{S+R} = S+R \subseteq \cl{S}$, yielding $S+R=\cl{S}$. Let us show that $\V \cap S + P_{\V\cap \cl{S}}(R) = \V\cap \cl{S}$ which is a closed subspace. The inclusion $\subseteq$ is clear. Let $x\in \V \cap \cl{S}$ be arbitrary, and since $x\in \cl{S}$ we can write it as $x=s_1 + r$ where $s_1\in S$ and $r\in R$. According to $S=\V^\bot \oplus (\V\cap S)$, there are $v_1\in \V^\bot$ and $v_2 \in \V\cap S$ satisfying $s_1 = v_1 + v_2$, and therefore $x=v_1+v_2+r$. Now, since $x, v_2 \in \V\cap \cl{S}$ and $v_1\bot \V\cap\cl{S}$, from $x=v_1+v_2+r$ we conclude that $x-v_2$ is exactly $P_{\V\cap \cl{S}} r$, hence $x = v_2 + P_{\V\cap \cl{S}} r \in \V \cap S + P_{\V\cap \cl{S}}(R)$. This proves that the opposite inclusion also holds.

Now assume that  $\V\cap S + P_{\V\cap \cl{S}} (R)$ is a closed subspace. From $\V^\bot \subseteq S$ we derive $\cl{\V\cap S} = \V\cap \cl{S}$ (see for instance \cite[Lemma 2.2]{Djikic2}), and so $\V\cap S + P_{\V\cap \cl{S}} (R)$ being closed implies $\V\cap S + P_{\V\cap \cl{S}} (R) = \V\cap \cl{S}$. Let us prove that $S+R=\cl{S}$, where we only need to show that $\cl{S}\subseteq S+R$. The subspace $\cl{S}$ is equal to $\V^\bot \oplus (\V\cap \cl{S}) = \V^\bot \oplus ( \V\cap S + P_{\V\cap \cl{S}} (R))$, so if we prove that $ P_{\V\cap \cl{S}} (R)\subseteq S+R$ the assertion follows. Note that $(\V\cap \cl{S})^\bot =  \V^\bot \oplus S^\bot$ and for arbitrary $r\in R$ we have $r- P_{\V\cap \cl{S}}r$ is both in $\cl{S}$ and in $\V^\bot \oplus S^\bot$. Hence, $r- P_{\V\cap \cl{S}}r \in \V^\bot$, showing that $P_{\V\cap \cl{S}}r \in S+R$. Accordingly, $P_{\V\cap \cl{S}} (R)\subseteq S+R$, and so $S+R=\cl{S}$.
 \end{proof}

 \begin{theorem}\label{djt1} Let $A,B\in \B(\H,\K)$ and $\M,\N$ two closed subspaces of $H$. There exists a closed operator $C$ with the domain $\D(C)=\M+\N$ such that $C$ coincides with $A$ on $\M$ and with $B$ on $\N$ if and only if $\R(A^*-B^*)+\M^\bot + \N^\bot =\cl{\M^\bot + \N^\bot}$.
 \end{theorem}
\begin{proof}
Both of the statements imply that $A$ and $B$ coincide on $\M\cap \N$, or equivalently $\cl{R(A^*-B^*)}\subseteq \cl{\M^\bot + \N^\bot}$, so we can assume that this condition is satisfied by $A$ and $B$ and consider the operator $C_{\M,\N}(A,B)$. We should in fact prove that $C_{\M,\N}(A,B)$ is closed if and only if $\R(A^*-B^*)+\M^\bot + \N^\bot$ is closed.

From Lemma \ref{djl4} we see that $C_{\M,\N}(A,B)$ is closed if and only if the restriction of $(A-B)(\proj)^\bot + B$ to $\M+\N$ is a closed operator. From Lemma \ref{djl9} we can conclude that this happens if and only if $(A-B)(\proj)^\bot$ is a closed operator which is according to Lemma \ref{djl3} equivalent to $\M\cap(\M^\bot + \N^\bot) + P_{\M\cap\cl{\M^\bot + \N^\bot}}( \R(A^*-B^*))$ being closed. If we apply the result of Lemma \ref{djl5} denoting by $S=\M^\bot + \N^\bot$, $R=\R(A^*-B^*)$, and $\V=\M$, we obtain that $\M\cap(\M^\bot + \N^\bot) + P_{\M\cap\cl{\M^\bot + \N^\bot}} (\R(A^*-B^*))$ is closed if and only if $\M^\bot + \N^\bot + \R(A^*-B^*)$ is closed, which completes the proof.
\end{proof}

We summarize by answering the questions from the beginning of Section \ref{djsec1}. If operators $A$ and $B$ coincide on $\M\cap \N$, that is if $C_{\M,\N}(A,B)$ is well-defined, then:

\begin{itemize}
  \item $C_{\M,\N}(A,B)$ is bounded if and only if $\R(A^*-B^*)\subseteq \M^\bot + \N^\bot$;

  \item $C_{\M,\N}(A,B)$ is closable if and only if $(A^*-B^*)^{-1}(\M^\bot + \N^\bot)$ is dense;

  \item $C_{\M,\N}(A,B)$ is closed if and only if $\R(A^*-B^*) + \M^\bot + \N^\bot$ is closed.
\end{itemize}

\section{Miscellaneous}
\subsection{The operator equation $B=XA$}

Recall a famous theorem by Douglas \cite{Douglas} which relates the range inclusion, (left) factorization and majorization for two bounded Hilbert space operators.

\begin{theorem}[\cite{Douglas}]\label{djtDouglas}
Let $\H,\K$ and $\L$ be Hilbert spaces, $S\in \B(\H,\K)$ and $T\in \B(\F,\K)$. The following statements are equivalent:
\begin{itemize}
  \item[(i)] $\R(T)\subseteq \R(S)$;
  \item[(ii)] There exists $X\in \B(\F,\H)$ such that $T=SX$;
  \item[(iii)] There exists $\lambda>0$ such that  $TT^*\leq \lambda^2 SS^*$.

\end{itemize}
\end{theorem}


Douglas' theorem was generalized in several directions (see for example \cite{Barnes} for a treatment on Banach spaces, and the recent paper \cite{Manuilov} and references therein for similar results in the Hilbert $C^*$-module setting).
One interpretation of this theorem is that the operator equation $T=SX$ has a bounded solution as soon as it has some linear solution $X_0 : \F \to \H$. This is not true for the dual equation: $B=XA$. Namely, if $B=X_0 A$ for some operator $X_0$, then we can only guarantee that there exists a semiclosed solution of $B=XA$, as demonstrated by the following theorem.

\begin{theorem}\label{djt4}
Let $\H, \K$ and $\F$ be Hilbert spaces and $A\in \B(\F,\H)$, $B\in \B(\F,\K)$. There exists an operator $X_0 : \H\to \K$ satisfying $B=X_0A$ if and only if the equation $B=XA$ has a solution in $\SC(\H,\K)$. Moreover:
\begin{itemize}
  \item[1.] The equation $B=XA$ has a solution $X\in \B(\H,\K)$ if and only if $\R(B^*)\subseteq \R(A^*)$;
  \item[2.] The equation $B=XA$ has a closed solution $X:\D(X)\to \K,\  \R(A)\subseteq \D(X)\subseteq \H$ if and only if $(B^*)^{-1}(\R(A^*))$ is dense in $\K$;
  \item[3.] The equation $B=XA$ has a closed solution $X:\D(X)\to \K,\  \R(A)= \D(X)$  if and only if $\R(A^*)+\R(B^*) = \cl{\R(A^*)}$.
\end{itemize}
\end{theorem}
\begin{proof}
If $X_0A = B$ for some operator $X_0 : \H\to \K$, then necessarily $\N(A)\subseteq \N(B)$, and therefore there is a $T\in \SC(\H,\K)$ satisfying $B=TA$, namely $T=B/A$. The opposite implication is immediate, since any solution of $B=XA$ can be extended to an operator defined on $\H$.

All the statements in 1. 2. and 3. imply $\N(A)\subseteq \N(B)$, and so the desired equivalences follow directly from Lemma \ref{djl11}, given that any solution of $B=XA$ has to be an extension of $B/A$. Of course, 1. also follows from Douglas' theorem.
\end{proof}

In the previous theorem we related the existence of a solution of $B=XA$ having a given continuity property with a ``range inclusion" condition, analogous to (i) in Theorem \ref{djtDouglas}. In the next theorem we give majorization conditions, analogous to (iii) in Theorem \ref{djtDouglas}.

\begin{theorem}\label{djt5}
Let $\H, \K$ and $\F$ be Hilbert spaces and $A\in \B(\F,\H)$, $B\in \B(\F,\K)$. Consider the following conditions for $A$ and $B$.
\begin{itemize}
  \item[(B)] For every sequence $(x_n)_{n\in \mathbb{N}} \subseteq \H$ the following implication holds: $$\mbox{If}\ Ax_n \to 0\ \mbox{then:}\ Bx_n\to 0.$$
  \item[(Ca)]  For every sequence $(x_n)_{n\in \mathbb{N}} \subseteq \H$ the following implication holds: $$\mbox{If}\ Ax_n \to 0\ \mbox{and}\ Bx_n \to x\  \mbox{then:}\ x= 0.$$
  \item[(C)]  For every sequence $(x_n)_{n\in \mathbb{N}} \subseteq \H$ the following implication holds: $$\mbox{If}\ Ax_n \to x\ \mbox{and}\ Bx_n \to y\  \mbox{then:}\ x\in \R(A)\ \mbox{and}\ BA^\dag x = y.$$
\end{itemize}
Conditions (B), (Ca) and (C) are equivalent to, respectively, statements from 1. 2. and 3. in Theorem \ref{djt4}.
\end{theorem}

\begin{proof}
For (B) see \cite[Proposition 3]{Barnes}, and for (Ca) see \cite[Lemma 2.3]{Izumino2} (having in mind that (Ca) implies $\N(A)\subseteq \N(B)$). Note also that (C) is equivalent to $B/A$ being well-defined and closed. Indeed, if (C) holds, then $\N(A)\subseteq \N(B)$ and using \eqref{djeq11} we see that $B/A$ is closed. Conversely, again using \eqref{djeq11} we have that (C) is satisfied. This proves that (C) is equivalent to the statements from 3. in  Theorem \ref{djt4}.
\end{proof}

We now refer to the setting of our problem. We will use the notation as in Section \ref{djsec1}, and for the sake of brevity, we will only give results about $C_{\M,\N}(A,B)$ being (well-defined and) bounded.

\begin{theorem}\label{djt6} The following statements are equivalent:
\begin{itemize}
\item[(i)] There exists a bounded operator $C:\H\to \K$ which coincides with $A$ on $\M$ and with $B$ on $\N$;
\item[(ii)] The equation $A-B = X (P_{\M^\bot} + P_{\N^\bot})^{1/2}$ has a solution in $\B(\H,\K)$;
\item[(iii)] The equation $(A-B)P_\M = X P_{\N^\bot}P_\M$ has a solution in $\B(\H,\K)$.
\end{itemize}
\end{theorem}
\begin{proof}
To prove that (i) and (ii) are equivalent we just use Theorems \ref{06djt1} and \ref{djtDouglas} and the fact that $\R((P_{\M^\bot} + P_{\N^\bot})^{1/2}) = \M^\bot + \N^\bot$. To see that these statements are equivalent to (iii), first we note that they all imply $\M\cap \N\subseteq \N(A-B)$, and then we can proceed, for example, in the following way: from Lemma \ref{djl4} we have that $C_{\M,\N}(A,B)$ is bounded if and only if $(A-B)(\proj)^\dag$ is bounded, which is by Lemma \ref{djl3} equivalent to $\R(P_{\N(\proj)^\bot} (A-B)^*)\subseteq \R((\proj)^*)$. Therefore, $C_{\M,\N}(A,B)$ is bounded if and only if the equation $(A-B)P_{\N(\proj)^\bot} = X \proj$ has a solution in $\B(\H,\K)$. Since $\N(\proj)^\bot = \M\ominus(\M\cap \N)$ and $\M\cap \N\subseteq \N(A-B)$, the last equation is equivalent to $(A-B)P_\M = X P_{\N^\bot}P_\M$.
\end{proof}

As a corollary, we give a metric test for checking whether $C_{\M,\N}(A,B)$ is bounded.

\begin{corolary}\label{djc1}
If $\M\cap \N=\{0\}$ then $C_{\M,\N}(A,B)$ is bounded if and only if
\begin{equation}
\label{djeq12}
\sup_{\substack{m\in \M \\ \|m\|=1}} \frac{\|(A-B)m\|^2}{1-\|P_\N m\|} < \infty.
\end{equation}
\end{corolary}

\begin{proof}
From statement (iii) in Theorem \ref{djt6} and Douglas' theorem we can conclude that $C_{\M,\N}(A,B)$ is bounded if and only if $$\| (A-B)P_\M x\|^2 \leq \lambda^2 \|P_{\N^\bot} P_\M x\|^2, \mbox{ for some }\lambda>0 \mbox{ and every }x\in \H.$$ The last statement can be restated as $\sup_{m\in \M\setminus\{0\}} \|(A-B) m\|^2 / \|P_{\N^\bot} m \|^2 <\infty$, or after renorming and using $\|P_{\N^\bot} m\|^2 = \|m\|^2 - \|P_{\N} m\|^2$, as
\begin{equation}\label{djeq13}
  \sup_{\substack{m\in \M \\ \|m\|=1}} \frac{\|(A-B)m\|^2}{1-\|P_\N m\|^2} < \infty.
\end{equation}
Note that the denominator in  \eqref{djeq13} is equal to $(1-\|P_\N m\|)(1+\|P_\N m\|)$ and $(1+\|P_\N m\|)^{-1} \in (1/2, 1]$, for $\|m\|=1$. Thus, \eqref{djeq13} is equivalent to \eqref{djeq12}.
\end{proof}

Corollary \ref{djc1} could also be useful if the intersection $\M\cap \N$ is non-trivial, but $A$ and $B$ coincide on it, given that $C_{\M,\N}(A,B)$ is bounded if and only if $C_{\M\ominus(\M\cap \N), \N\ominus (\M\cap \N)}(A,B)$ is bounded.

\begin{rem}\label{fdjr}
A well-known angle test says that $\M\dotplus \N$ is closed if and only if $\sup \{|\sk{x}{y}|\ :\ x\in \M,\ y\in \N,\ \|x\|=\|y\|=1\} < 1$ (see \cite{D}). This is contained in Corollary \ref{djc1}. Namely, $\M\dotplus\N$ is closed if and only if $C_{\M,\N}(I,0)$ is bounded (see Lemma \ref{djl7}). We also know that for $m\in \H$ such that $\|m\|=1$, $\|P_\N m\| = \max\{|\sk{m}{n}|\ :\ n\in \N,\ \|n\|=1\}$. Putting together these conclusions, from Corollary \ref{djc1} we get that $\M \dotplus \N$ is closed if and only if $\inf \{1-|\sk{m}{n}|\ :\ m\in \M, n\in \N, \|m\|=\|n\|=1\} > 0$, which is exactly the angle test. \hfill $\diamond$
\end{rem}

\subsection{Canonical decomposition for two subspaces}

By now it is obvious that the continuity properties of $C_{\M,\N}(A,B)$ and $C_{\M,\N}(A-B,0)$ coincide, and in fact they are encoded in the relation between two operators restricted on $\M$: the operator $A-B|_\M$ and the operator $P_{\N^\bot}|_\M$ (see e.g. statement (iii) of Theorem \ref{djt6}). Halmos' two projections theorem \cite{Halmos1} gives us a canonical approach to problems regarding the relationship of two orthogonal projections, and a clear description of the restriction $P_{\N^\bot}|_\M$. Thus, we wish to describe our results in the language of this far-reaching approach. We also refer the reader to the survey paper \cite{Bottcher1} for a detailed reference on these matters, and to \cite{Bottcher2} for even more examples.

If $\M$ and $\N$ are two closed subspaces in a Hilbert space $\H$, then the subspaces $\M$ and $\M^\bot$ can be decomposed with respect to $\N$ as: $\M=(\M\cap\N) \oplus (\M \cap \N^\bot) \oplus \M_0$ and $\M^\bot =( \M^\bot\cap\N) \oplus (\M^\bot\cap \N^\bot) \oplus \M_1$, and so $\H=(\M\cap\N) \oplus (\M \cap \N^\bot) \oplus (\M^\bot\cap\N) \oplus (\M^\bot\cap \N^\bot) \oplus \M_0 \oplus \M_1$. The subspace $\M_0 \oplus \M_1$ is invariant for $P_\N$. In fact, Halmos' two projections theorem says that the subspaces $\M_0$ and $\M_1$ are isomorphic and that, if they are not trivial, there is a unitary operator $R:\M_0 \to \M_1$, and a positive contraction $S\in \B(\M_0)$, such that $S$ and $1-S$ are injective, and that  the operator matrix of $P_{\N}$ with respect to the decomposition $\M_0\oplus \M_1$ has the following form: $$P_{\N}|_{\M_0\oplus \M_1} = \begin{bmatrix} 1-S^2 & S\sqrt{1-S^2}R^* \\ RS\sqrt{1-S^2} & RS^2R^* \end{bmatrix}.$$
The operator $S$ carries all the important information about the relationship between $\M$ and $\N$, so it should also carry the information about boundedness (and other properties) of $C_{\M,\N}(A,B)$. In the following theorem we see that it does.

\begin{theorem}
\label{djt7}
Let  $A$ and $B$ coincide on $\M\cap \N$. If the subspace $\M_0$ is trivial, then $C_{\M,\N}(A,B)$ is bounded. Otherwise, let $D:\M_0 \to \K$ be the restriction of $A-B$ to $\M_0$. Then $C_{\M,\N}(A,B)$ is bounded if and only if $\R(D^*)\subseteq \R(S)$.
\end{theorem}
\begin{proof}
According to Theorem \ref{djt6}, $C_{\M,\N}(A,B)$ is bounded if and only if the equation $(A-B)P_\M = XP_{\N^\bot}P_\M$ has a solution $X\in \B(\H,\K)$. We directly deduce that, if $\M_0=\{0\}$ this equation surely has a solution, while if $\M_0\not = \{0\}$, the equation has a solution if and only if the equation $D = \begin{bmatrix} X_1\ X_2\end{bmatrix} \begin{bmatrix} S^2 \\ RS\sqrt{1-S^2}\end{bmatrix}$ has a solution $\begin{bmatrix} X_1\ X_2\end{bmatrix} \in \B(\M_0\oplus \M_1, \K)$, or in other words if and only if $\R(D^*)\subseteq \R(\begin{bmatrix} S^2 & \sqrt{1-S^2}SR^*\end{bmatrix})$. The proof is completed by showing that $$\R(\begin{bmatrix} S^2 & \sqrt{1-S^2}SR^*\end{bmatrix}) = \R(S),$$ and this is easily seen to be true, for example, from $$\R(\begin{bmatrix} S^2 & \sqrt{1-S^2}SR^*\end{bmatrix}) = \R((\begin{bmatrix} S^2 & \sqrt{1-S^2}SR^*\end{bmatrix}\cdot \begin{bmatrix} S^2 \\ RS\sqrt{1-S^2} \end{bmatrix})^{1/2}).$$
\end{proof}

Appropriate statements for closability and closedness of $C_{\M,\N}(A,B)$ can be proved in a similar fashion. In fact, the continuity properties of $C_{\M,\N}(A,B)$ coincide to those of the quotient $D/S$. Example 3.2 in \cite{Bottcher1} tells us that $\M+\N$ is closed if and only if $\M_0=\{0\}$ or $S$ is invertible. Of course, this agrees with Theorems \ref{djt7} and \ref{jt8} as well.

\subsection{Coherent pairs and star partial order}

We now describe one application of our results. It is concerned with the so called star partial order, a notion introduced by Drazin \cite{Drazin}, which also appeared in several other instances. For example, it was independently suggested by Gudder \cite{Gudder} as a natural order for bounded quantum observables, but it is also intimately related to a much older notion of *-orthogonality by Hestenes \cite{Hestenes}. We will give a brief introduction here, and we refer the reader to \cite{Djikic, Djikic6} for more information and further  references.

The \textit{ star partial order} $\zv{\leq}$ is defined as
\begin{equation}\label{djeq14} \mbox{For}\ A,C\in\B(\H,\K)\quad \quad A\zv{\leq} C \quad \quad \Leftrightarrow \quad \quad AA^*=CA^* \quad \& \quad A^*A=A^*C.\end{equation}
Equivalently, and more importantly, $A\zv{\leq} C$ if and only if $A$ and $C$ coincide on $\N(A)^\bot$ and $C(\N(A))\subseteq \N(A^*)$. The problem we are concerned with is: if $A,B\in \B(\H,\K)$ what are necessary and sufficient conditions for the star-supremum of $A$ and $B$ to exist? 

In addressing this problem, the first author introduced the notion of \textit{coherent pairs} in \cite{Djikic}. For two closed subspaces $\M,\N\subseteq \H$, and $A,B\in \B(\H,\K)$, the pairs $(A,\M)$ and $(B,\N)$ are called coherent if $C_{\M,\N}(A,B)$ is bounded (in the notation of Section \ref{djsec1}). The first thing we wish to highlight in this section is that we now have a characterization of coherent pairs.
\begin{corolary}
The pairs $(A,\M)$ and $(B,\N)$ are coherent if and only if $\R(A^*-B^*)\subseteq \M^\bot + \N^\bot$.
\end{corolary}

The relationship with the star-supremum is obvious: if $A\zv{\leq} C$ and $B \zv{\leq} C$, then the operator $C_{\N(A)^\bot, \N(B)^\bot}(A,B)$ is well-defined and bounded. In fact, the following characterization is proved in \cite{Djikic}.

\begin{lemma}[\cite{Djikic}]\label{fdjl1} The star-supremum of $A$ and $B$ exists if and only if  $C_{\N(A)^\bot, \N(B)^\bot}(A,B)$ and $C_{\N(A^*)^\bot,\N(B^*)^\bot}(A^*,B^*)$ are both well-defined and bounded, and moreover $C(A,B) = C(A^*,B^*)^*$, where $C(A,B)\in \B(\H,\K)$ is the bounded operator which is the continuous extension of the operator $C_{\N(A)^\bot,\N(B)^\bot}(A,B)$ on $\cl{\N(A)^\bot+\N(B)^\bot}$ and the null-operator on $\N(A)\cap\N(B)$, and similarly for the operator $C(A^*,B^*)$.
\end{lemma}

By assuming $A\zv{\leq} C$ and $B \zv{\leq} C$, and by canceling out $C$ from \eqref{djeq14}, we obtain a necessary condition for $A$ and $B$ to have the star-supremum, namely
\begin{equation}\label{djeq15} AA^*B=AB^*B \quad \mbox{and} \quad BA^*A=BB^*A.\end{equation} The question was raised in \cite{Djikic}: is this a sufficient condition as well? 
In fact, it was shown that this question reduces to whether \eqref{djeq15} imply that $C_{\N(A)^\bot,\N(B)^\bot}(A,B)$ is bounded, see \cite[Theorem 3.1, Theorem 3.2]{Djikic}. We are still not able to answer this question, but using Theorem \ref{06djt1} we can readily check that, if \eqref{djeq15} holds, then $C_{\N(A)^\bot,\N(B)^\bot}(A,B)$ is well-defined and closable (this was proved differently in \cite{Djikic}). 
Obviously, if $\N(A)^\bot + \N(B)^\bot$ is closed then \eqref{djeq15} would imply $C_{\N(A)^\bot,\N(B)^\bot}(A,B)$ being bounded, and consequently $A$ and $B$ having the star-supremum. This was noted in \cite{Djikic}, but the results of this paper can lead a bit further, in the sense that we can give another, less obvious, sufficient condition. 

\begin{lemma}\label{fdjl2}
Let $A,B\in \B(\H,\K)$ such that $\R(A-B)$ is closed. The star-supremum of $A$ and $B$ exists if and only if \eqref{djeq15} holds.
\end{lemma}
\begin{proof}
If \eqref{djeq15} holds, we have: $A(A^*-B^*)B=0=B(A^*-B^*)A$. Therefore, $(A^*-B^*)(\R(B))\subseteq \N(A)$ and $(A^*-B^*)(\R(A))\subseteq \N(B)$. Consequently, $\R((A^*-B^*)(A-B))\subseteq \R((A^*-B^*)A) + \R((A^*-B^*)B) \subseteq \N(A)+\N(B)$. Since $\R(A^*-B^*)$ is closed, $\R((A^*-B^*)(A-B))=\R(A^*-B^*)$, and so $\R(A^*-B^*)\subseteq \N(A)+\N(B)$, which according to Theorem \ref{06djt1} means that $C_{\N(A)^\bot,\N(B)^\bot}(A,B)$ is bounded. Thus, using \cite[Theorem 3.2]{Djikic}, we see that the star-supremum of $A$ and $B$ exists. The opposite implication is already explained.
\end{proof}

\section*{Acknowledgment}

The authors are supported by Grant No. 174007 of the Ministry of Education, Science and Technological Development, Republic of
Serbia.

\bibliography{Literatura}
\bibliographystyle{plain}

\end{document}